\documentclass[11pt]{amsart}
\usepackage[british]{babel}

\usepackage{amssymb,latexsym}
\usepackage{comment}
\usepackage{titletoc}
\usepackage{amsfonts, amsmath, amsthm, amssymb, fancybox, graphicx, color, times, babel}
\usepackage[utf8]{inputenc}
\usepackage{multirow}
\usepackage{enumerate}
\usepackage{hyperref}
\usepackage[usenames,dvipsnames]{xcolor}
\usepackage[all,cmtip]{xy}
\usepackage[normalem]{ulem}
\usepackage{xr-hyper}
\usepackage{hyperref}
\usepackage{tikz-cd}

\newtheorem {thm}{Theorem}
\newtheorem* {thm*}{Theorem}

\newtheorem* {cor*}{Corollary}
\newtheorem {lem}[thm]{Lemma}

\newtheorem {prop}[thm]{Proposition}
\newtheorem {rem}[thm]{Remark}

\theoremstyle{definition}
\newtheorem {defi}[thm]{Definition}
\newtheorem {exa}[thm]{Example}
\newtheorem* {conj*}{Conjecture}

\newtheorem* {quest*}{Question}

\DeclareMathOperator{\End}{End}

\DeclareMathOperator{\Gal}{Gal}
\DeclareMathOperator{\Hom}{Hom}

\DeclareMathOperator{\Frob}{Fr}

\DeclareMathOperator{\Aut}{Aut}

\DeclareMathOperator{\Res}{Res}

\DeclareMathOperator{\ad}{ad}
\DeclareMathOperator{\Tr}{Tr}
\DeclareMathOperator{\ST}{ST}

\DeclareMathOperator{\USp}{USp}

\DeclareMathOperator{\Sym}{Sym}
\DeclareMathOperator{\M}{M}

\newcommand{\Q}{\mathbb{Q}}

\newcommand{\Z}{\mathbb{Z}}
\newcommand{\R}{\mathbb{R}}

\newcommand{\Qbar}{{\overline \Q}}
\newcommand{\Ebar}{\overline E}

\newcommand{\Xc}{\mathcal X}

\newcommand{\p}{\mathfrak{p}}

\newcommand{\triv}{\mathbf{1}}
\DeclareMathOperator{\GL}{GL}

\DeclareMathOperator{\Ind}{Ind}

\renewcommand{\geq}{\geqslant}
\renewcommand{\leq}{\leqslant}

\newcommand{\SU}{\mathrm{SU}}
\newcommand{\Ab}{\textbf{A}}
\newcommand{\Bb}{\textbf{B}}
\newcommand{\Cb}{\textbf{C}}
\newcommand{\Db}{\textbf{D}}
\newcommand{\Eb}{\textbf{E}}
\newcommand{\Fb}{\textbf{F}}

\DeclareMathOperator{\C}{\mathbb C}
\newcommand*{\stgroup}[2][]{\href{https://www.lmfdb.org/SatoTateGroup/#2}{{\ifx&#1& #2 \else #1 \fi}}}

\newcommand{\arXiv}[2]{\href{https://arxiv.org/abs/#1}{arXiv:#1v#2}}

\title[A local-global principle for polyquadratic twists of abelian surfaces]{A local-global principle for polyquadratic twists\\ of abelian surfaces}

\author{Francesc Fit\'e}

\address{Departament de matem\`atiques i inform\`atica and Centre de recerca matem\`atica,
Universitat de Barcelona,
Gran via de les Corts Catalanes 585, 08007 Barcelona, Catalonia.}
\email{ffite@ub.edu}
\urladdr{http://www.ub.edu/nt/ffite/}

\author{Antonella Perucca}

\address{Department of Mathematics, University of Luxembourg: 6, avenue de la Fonte, L-4364, Esch-sur-Alzette, Luxembourg.}
\email{antonella.perucca@uni.lu}
\urladdr{https://www.antonellaperucca.net/}

\begin{document}

\begin{abstract}
We say that two abelian varieties $A$ and $A'$ defined over a field $F$ are polyquadratic twists if they are isogenous over a Galois extension of $F$ whose Galois group has exponent dividing $2$. Let $A$ and $A'$ be abelian varieties defined over a number field $K$ of dimension $g\geq 1$. In this article we prove that, if $g\leq 2$, then $A$ and $A'$ are polyquadratic twists if and only if for almost all primes $\p$ of $K$ their reductions modulo $\p$ are polyquadratic twists. We exhibit a counterexample to this local-global principle for $g=3$. This work builds on a geometric analogue by Khare and Larsen, and on a similar criterion for quadratic twists established by Fité, relying itself on the works by Rajan and Ramakrishnan.
\end{abstract}

\maketitle

\section{Introduction}

Let $K$ denote a number field and let $A$ and $A'$ be abelian varieties defined over~$K$ of dimension $g\geq 1$. Call $\Sigma_K$ the set of nonzero prime ideals of the ring of integers $\mathcal O_K$ of $K$. Only finitely many primes in $\Sigma_K$ are of bad reduction for $A$ and $A'$. In this article, by ``almost all $\p\in \Sigma_K$" we will mean all primes of $\Sigma_K$ outside a zero density subset containing the finite set of primes of bad reduction.

Faltings' isogeny theorem \cite{Fal83} asserts that $A$ and $A'$ are isogenous if and only if their reductions $A_\mathfrak p$ and $A'_\mathfrak p$ modulo $\p$ are isogenous for almost all $\mathfrak p\in \Sigma_K$.
We say that two abelian varieties are \emph{geometrically isogenous} if their base changes to an algebraic closure of their field of definition are isogenous. Building on Faltings' isogeny theorem and a result by Pink \cite{Pin98}, Khare and Larsen \cite[Thm.\ 1]{KL20} have shown that $A$ and $A'$ are geometrically isogenous if and only if their reductions $A_\mathfrak p$ and $A'_\mathfrak p$ modulo $\p$ are geometrically isogenous for almost all $\mathfrak p\in \Sigma_K$ (in fact, they show that it suffices to require this property for a set of primes of sufficiently large density). 

In the context of the result of Khare and Larsen, Faltings' isogeny theorem  asserts that the property of $A$ and $A'$ admitting an isogeny defined over $K$ can be read purely from local data. 
There is a variety of results in the literature studying whether certain global properties of isogenies can be inferred from local information. For example, when $g=1$ and $m$ is a positive integer, the articles \cite{Sut12}, \cite{Ann14}, \cite{Vog20} study the problem of characterizing the existence of a degree $m$ isogeny of $A$ in terms of the existence of degree $m$ isogenies for $A_\p$ for almost all $\p \in \Sigma_K$. The analogous problem when $g=2$ has been recently explored by \cite{Ban21} and \cite{LV22}. Rather than on the degree of the isogeny, in the present article we focus on its field of definition.

We say that $A$ and $A'$ are \emph{quadratic twists} if $A'$ is isogenous to $A_\chi$ over $K$, where $A_\chi$ denotes the twist of $A$ by a quadratic character $\chi$ of $G_K$, regarded as an element in $H^1(G_K, \{\pm 1\})\subseteq H^1(G_K, \Aut(A_\Qbar))$. We say that $A$ and $A'$ are \emph{locally quadratic twists} if~$A_\mathfrak p$ and~$A'_\mathfrak p$ are quadratic twists for almost all $\mathfrak p\in \Sigma_K$. In the spirit of the previously mentioned results, the main theorem of \cite{Fit22} shows that, for $g\leq 3$, $A$ and $A'$ are quadratic twists if and only if they are locally quadratic twists. Moreover, a counterexample to this local-global principle in dimension $g=4$ is presented.

Notice that the condition of being quadratic twists is finer than that of being isogenous over a quadratic extension: if $A=E\times E_\chi$ and $A'=E^2$, where $E$ is an elliptic curve defined over~$K$ without CM and $\chi$ is a quadratic character, then $A$ and $A'$ are isogenous over a quadratic extension, but $A$ and $A'$ are not quadratic twists according to the definition in the previous paragraph. For similar reasons, the condition of being isogenous over a quadratic extension is not a notion that one can expect to satisfy a local-global principle beyond dimension $2$. Indeed, if $A=E_\chi\times E_\psi$ and $A'=E^2$, where $E$ is an elliptic curve defined over~$K$ without CM and $\chi$ and $\psi$ are distinct quadratic characters, then $A_\p$ and $A'_\p$ are isogenous over a quadratic extension for almost all $\p\in \Sigma_K$, but $A$ and $A'$ only become isogenous over a biquadratic extension.
  
The discussion in the previous paragraph suggests that only the following weakening of the notion of being isogenous over a quadratic extension can aspire to satisfy a local-global principle.
We say that a Galois field extension is polyquadratic if its Galois group has exponent dividing~$2$. Note that a polyquadratic extension of a finite field is either trivial or quadratic. We say that two abelian varieties are \emph{polyquadratic twists} if their base changes to a polyquadratic extension
are isogenous. We say that $A$ and $A'$ are \emph{locally polyquadratic twists} if $A_\p$ and~$A'_\p$ are polyquadratic twists for almost all $\p\in \Sigma_K$.

It is easy to see (cf. Remark \ref{explanation}) that if $A$ and $A'$ are polyquadratic twists, then $A$ and $A'$ are locally polyquadratic twists.
The goal of this article is to study the converse of this implication. For $g=1$, this converse follows easily from results of Rajan \cite[Thm. B]{Raj98} and Ramakrishnan \cite[Thm. 2]{Ram00} (see Lemma \ref{lemma: dimension1}). 
The main result of this article is the following.

\begin{thm}\label{theorem: main}
Suppose that $A$ and $A'$ are abelian surfaces defined over a number field $K$. Then they are polyquadratic twists if and only if they are locally polyquadratic twists.
\end{thm}

To complement the above theorem, in \S\ref{section: counterexample} we exhibit two abelian threefolds that are locally polyquadratic twists but are not polyquadratic twists. In \S\ref{section: preliminaries}, we translate our problem into a problem about $\ell$-adic Galois representations and reinterpret the notions defined in this introduction in group theoretic terms. We also recall the main results of \cite{KL20} and \cite{Fit22}, and we prove some technical results that will be used in later sections. In \S\ref{section: squares} we carry out the proof of Theorem \ref{theorem: main} in the most complicated case, that is, when $A$ is either geometrically isogenous to the square of an elliptic curve or has geometric quaternionic multiplication. A crucial input in this case is the tensor decomposition of the Tate module $T_\ell(A)$ provided by \cite{FG22}. This remarkably allows us to translate our problem concerning $\ell$-adic representations of degree~$4$ into a problem concerning Artin representations of degree $2$ 
(see Theorem~\ref{thm: locpolytwists} for a solution of the latter problem). The proof of the remaining cases of Theorem \ref{theorem: main} takes place in \S\ref{section: proof}, where we benefit from the classification of Sato-Tate groups of abelian surfaces of \cite{FKRS12}. We remark however that our proof is independent of the Sato-Tate conjecture. 

\subsection*{Notation and terminology.} All algebraic extensions of $K$ are contained in some fixed algebraic closure $\Qbar$ of $\Q$, and we denote by $G_K$ the absolute Galois group $\Gal(\Qbar/K)$. For a field extension $L/K$, we write $A_L$ to denote the base change of $A$ from $K$ to $L$. Homomorphisms between abelian varieties are always implicitly understood to be defined over the field of definition. Given a representation $\varrho$, we call $\varrho^\vee$ its contragredient representation, $\ad(\varrho)\simeq \varrho \otimes \varrho^\vee$ its adjoint representation, and $\ad^0(\varrho)$ the subrepresentation of $\ad(\varrho)$ on the trace $0$ subspace. If $\varrho$ is a representation of $G_K$ with coefficients in a field $E$ and $L/K$ is an algebraic extension, we denote by $\varrho|_L$ the restriction of $\varrho$ to $G_L$; we say that $\varrho$ is absolutely irreducible if $\varrho\otimes \overline E$ is irreducible; and we say that $\varrho$ is strongly absolutely irreducible if $\varrho|_L$ is absolutely irreducible for every finite extension~$L/K$. 

\subsection*{Acknowledgements.} We thank Shiva Chidambaram for Example \ref{shiva}, and the anonymous referee for comments that led to improvements in the exposition of the results. Fit\'e thanks the organizers of the PCMI Research Program ``Number Theory Informed by Computation", which facilitated fruitful discussions with Shiva Chidambaram and David Roe. Fit\'e expresses his gratitude to the University of Luxembourg for its warm hospitality during the last week of May 2022, where discussions with Perucca that eventually led to the present work started.
Fit\'e was financially supported by the Ram\'on y Cajal fellowship RYC-2019-027378-I, the Simons Foundation grant 550033, and the Mar\'ia de Maeztu Program of excellence CEX2020-001084-M. 

\section{Preliminaries}\label{section: preliminaries}

\subsection{Group theoretic descriptions}\label{section: grouptheoretic}

Let $E$ be a topological field of characteristic 0, and fix an algebraic closure~$\Ebar$. Let $G$ be a compact topological group, and let $\varrho,\varrho': G\rightarrow \GL_r(E)$ be semisimple continuous group representations, where $r$ is a positive integer. 

\begin{defi}
We call $\varrho$ and $\varrho'$ \emph{locally quadratic twists} if for every $s\in G$ there exists $\epsilon_s\in \{\pm 1\}$ such that
$$
\det(1-\varrho(s)T)= \det(1-\epsilon_s\varrho'(s)T)\,.
$$
We call them \emph{locally polyquadratic twists} if for every $s\in G$ we have
$$
\det(1-\varrho(s^2)T)= \det(1-\varrho'(s^2)T)\,.
$$
\end{defi}

One can characterize the notions of being locally polyquadratic and locally quadratic twists in terms of relations satisfied by the alternating and symmetric powers of the given representations. When $r=4$, these relations acquire a particularly clean form.

\begin{lem}\label{lemma: gpcharac}
The representations $\varrho, \varrho'$ are locally polyquadratic twists if and only if
\begin{equation}\label{virtual}
\Sym^2 \varrho - \wedge^ 2\varrho \simeq \Sym^2 \varrho' - \wedge^ 2\varrho'
\end{equation}
as virtual representations.

If $r=4$, they  are locally quadratic twists if and only if
$$
\Sym^ 2 \varrho \simeq \Sym^ 2 \varrho'\qquad \text{and}\qquad \wedge^ 2 \varrho \simeq \wedge^ 2 \varrho'\,.
$$
\end{lem}
\begin{proof}
For the first assertion, it suffices to verify that both sides of \eqref{virtual} have the same virtual character.
Indeed, if for $s \in G$ we denote by $\alpha_i$ the eigenvalues of $\varrho(s)$, then we have
$$\Tr\Sym^2\varrho(s) - \Tr\wedge^2 \varrho(s)=\sum_i \alpha_i^2=\Tr\Sym^2\varrho'(s) - \Tr\wedge^2 \varrho'(s)\,.$$
Now consider the second assertion.
As $r=4$, we know that $\varrho$ and $\varrho'$ are locally quadratic twists if and only if we have 
\begin{equation}\label{equation: isomorphisms}
\varrho \otimes \varrho \simeq \varrho' \otimes \varrho',\quad \wedge^2 \varrho \simeq \wedge^ 2 \varrho'\,,\quad \varrho \otimes \wedge^ 3 \varrho \simeq \varrho' \otimes \wedge^ 3 \varrho'\,,\quad \det(\varrho) \simeq \det(\varrho')\,.
\end{equation}
Then it suffices to show that the third and fourth of the above isomorphisms follow from the first two. Taking determinants, the first (respectively, second) isomorphism implies $\det(\varrho)^8\simeq \det(\varrho')^ 8$ (respectively, $\det(\varrho)^3\simeq \det(\varrho')^3$) and we deduce $\det(\varrho) \simeq \det(\varrho')$. As $\varrho$ and $\varrho'$ are locally quadratic twists, we know $\varrho \otimes \varrho^ \vee\simeq \varrho' \otimes \varrho'^ \vee$. Considering that $\wedge^3 \varrho \simeq \det(\varrho)\otimes \varrho^\vee$ (and similarly for $\varrho'$), we obtain $\varrho \otimes \wedge^ 3 \varrho \simeq \varrho' \otimes \wedge^ 3 \varrho'$.\end{proof}

\begin{defi}\label{def-polyquadratic-twists}
We call $\varrho$ and $\varrho'$ \emph{quadratic twists} if $\varrho'\simeq \chi \otimes \varrho$ holds for some quadratic character~$\chi$ of $G$. We call them \emph{polyquadratic twists} if 
\begin{equation}\label{equation: polquadtwist}
\varrho \simeq \bigoplus_{i=1}^t\varrho_i\qquad\text{and}\qquad  \varrho' \simeq \bigoplus_{i=1}^t\varrho_i'
\end{equation}
holds for representations $\varrho_i,\varrho_i': G\rightarrow \GL_{n_i}(E)$ such that $\varrho_i'\simeq \chi_i\otimes \varrho_i$, where $n_i$ is a positive integer and $\chi_i$ is a quadratic character of $G$.
\end{defi}

\begin{rem}\label{used}
Suppose that $r=2$, and let $\varrho$ and $\varrho'$ be polyquadratic twists. Then either $\varrho$ and~$\varrho'$ are quadratic twists or there exist characters $\varphi_1,\varphi_2$ of $G$ and quadratic characters $\chi_1,\chi_2$ of~$G$ such that $\varrho\simeq \varphi_1 \oplus \varphi_2$ and $\varrho' \simeq \chi_1\varphi_1 \oplus \chi_1\varphi_2 $.
\end{rem}

We have the following characterization of polyquadratic twists.

\begin{prop}\label{proposition: polyquadratictwists}
The representations $\varrho,\varrho'$ are polyquadratic twists if and only if $\varrho|_H\simeq \varrho'|_H$ holds for some normal subgroup $H\subseteq G$ such that $G/H$ is a finite abelian group of exponent dividing~$2$.
\end{prop}
\begin{proof}
Supposing that $\varrho,\varrho'$ are polyquadratic twists, we may take for $H$ the intersection of the kernels of $\chi_1,\ldots, \chi_r$. For the converse implication, let $\theta$ be an irreducible constituent of~$\varrho$. By the argument in the proof of \cite[Thm.\ 3.1]{Fit12}, the representation $\theta$ is an irreducible constituent of
$\Hom_{H}(\theta,\varrho')\otimes \varrho'$, where $\Hom_{H}(\theta,\varrho')$ denotes the space of $H$-equivariant homomorphisms from $\theta$ to $\varrho'$.
Since $\Hom_{H}(\theta,\varrho')$, as a representation of $G/H$, is a sum of quadratic characters of $G/H$, there exists an irreducible constituent $\theta'$ of $\varrho'$ and a quadratic character $\chi$ of $G/H$ such that $\theta'\simeq \chi \otimes \theta$. We can apply the same argument to the complement of $\theta$ in $\varrho$ and the complement of $\theta'$ in $\varrho'$, and the proposition follows by iterating the above step.
\end{proof}

\begin{rem}\label{Ebar}
By the above proposition, in Definition \ref{def-polyquadratic-twists} we may replace $E$ by $\Ebar$.
\end{rem}

\begin{rem}\label{rem: locpol implies pol}
If $\varrho$ and $\varrho'$ are quadratic (respectively, polyquadratic) twists, then clearly they are locally quadratic (respectively, polyquadratic) twists.
\end{rem}

\begin{rem}\label{remark: factors}
Let $\varrho_i$ and $\varrho_i'$, for $i\in \{ 1,2,3\}$, be semisimple representations of $G$ satisfying
$$
\varrho_1\simeq \varrho_2\oplus \varrho_3\,,\qquad  \varrho_1'\simeq \varrho_2'\oplus \varrho_3'\,.
$$ 
Suppose that $\varrho_i$ and $\varrho_i'$ are polyquadratic twists (resp.\ locally polyquadratic twists) for all $i\in I$, where $I$ is a subset of $\{1,2,3\}$ with $|I|=2$. It then follows immediately from the definitions that $\varrho_i$ and $\varrho_i'$ are polyquadratic twists (resp.\ locally polyquadratic twists) for all $i$. The previous
property is not true in general if we replace ``polyquadratic" by ``quadratic".
\end{rem}

\begin{thm}[Ramakrishnan]\label{thm: Ramakrishnan}
Suppose that $r=2$. If $\varrho$ and $\varrho'$ are locally quadratic twists, then they are quadratic twists.
\end{thm}
\begin{proof}
The hypothesis is equivalent to having $\wedge^2 \varrho \simeq \wedge^2 \varrho'$ and $\varrho\otimes \varrho \simeq \varrho' \otimes \varrho'$. Since $\varrho^ \vee \otimes \wedge^2 \varrho \simeq \varrho$ (and the same holds for $\varrho'$) we obtain $\ad^0(\varrho)\simeq \ad^0(\varrho')$.
By \cite[Thm. B]{Ram00} (the result is stated for $\ell$-adic representations, but the proof is valid in our context) there exists a character $\chi$ of $G$ such that $\varrho'\simeq \chi \otimes \varrho$. Morever, $\chi$ must be quadratic because $\wedge^2 \varrho \simeq \wedge^2 \varrho'$. 
\end{proof}

\begin{rem}
The above result also holds for $r$ odd by \cite[Cor.\ 4.3]{Fit22}.
\end{rem}

\begin{exa}[Chidambaram]\label{shiva}
For $r=4$ and $G$ with \cite{GAP} identifier $\langle 12,1\rangle$, there are locally quadratic twists that are not quadratic twists. Indeed, let $\theta$ denote the only rational $2$-dimensional irreducible representation of $G$ up to isomorphism, let $\varepsilon$ be the rational nontrivial character of $G$, and let $\chi$ denote any character of $G$ of order $4$. It is a straighforward computation to verify that $\triv\oplus \varepsilon \oplus \chi\otimes \theta$ and $\chi\oplus \varepsilon\chi \oplus \theta$ are locally quadratic twists, but are not quadratic twists.
\end{exa}

\begin{thm}\label{thm: locpolytwists}
Suppose that $r=2$. If $\varrho$ and $\varrho'$ are locally polyquadratic twists, then they are polyquadratic twists.
\end{thm}
\begin{proof}
By Remark \ref{Ebar} we may assume that $E$ is algebraically closed. For every $s\in G$, let $\alpha_s, \beta_s\in E$ be such that
$$
\det(1-\varrho(s)T)=(1-\alpha_s T)(1-\beta_s T)\,.
$$
By assumption, there exist $\psi_s,\varphi_s\in \{\pm 1 \}$ such that
$$
\det(1-\varrho'(s)T)=(1-\psi_s\alpha_s T)(1-\varphi_s\beta_s T)\,.
$$
The map $\varepsilon: G\rightarrow \{\pm 1\}$ mapping $s\in G$ to 
\begin{equation}\label{epsilon}
\varepsilon(s):=\frac{\psi_s}{\varphi_s}=\frac{\det(\varrho')}{\det(\varrho)}(s)
\end{equation}
is a quadratic character.

Suppose that $\varepsilon$ is trivial. Then $\varrho$ and $\varrho'$ are locally quadratic twists, and hence by Theorem \ref{thm: Ramakrishnan} they are quadratic twists.

Now suppose that $\varepsilon$ is nontrivial. We will show in the next paragraph that both $\varrho$ and $\varrho'$ are reducible. Assuming this, there exist characters $\chi_1, \chi_2, \chi_1', \chi_2':G_K\rightarrow E^ \times$ such that $\varrho\simeq \chi_ 1 \oplus \chi_2$ and $\varrho'\simeq \chi_ 1' \oplus \chi_2'$.
By comparing traces, one immediately verifies that
$$
\frac{\chi_1}{\chi_2}\oplus \frac{\chi_2}{\chi_1} \simeq \frac{\varepsilon \chi'_1}{\chi'_2}\oplus \frac{\varepsilon \chi'_2}{\chi'_1}\,.
$$
Up to renaming $\chi_1'$, $\chi_2'$, we may assume $\chi_1/\chi_2 =\varepsilon \chi'_1/\chi'_2$. Also, by \eqref{epsilon} we have that $\chi_1\chi_2= \varepsilon \chi_1'\chi_2'$. Multiplying the last two equations we get $(\chi_1)^2=\varepsilon^2(\chi'_1)^2=(\chi'_1)^2$, which implies that $\chi_1$ and $\chi'_1$ differ by a quadratic character. The same holds for $\chi_2$ and $\chi'_2$. Hence $\varrho$ and $\varrho'$ are polyquadratic twists. 

We now prove that $\varrho'$ is reducible (the proof for $\varrho$ being the same) by showing that the multiplicity of the trivial representation $\triv$ in $\ad^0(\varrho')$ is positive (observe that $\ad^0(\varrho')$ is semisimple by \cite[p. 88]{Che54}). A straightforward computation shows that
$$
\ad^0(\varrho')\simeq \triv - \varepsilon + \varepsilon \otimes \ad^0(\varrho)
$$
as virtual representations.
Hence the multiplicity of $\triv$ in $\ad^0(\varrho')$ satisfies 
$$
\langle \triv, \ad^0(\varrho')\rangle = 1+ \langle \triv, \varepsilon \otimes \ad^ 0(\varrho)\rangle\geq 1 \,,
$$
where we have used that $\langle \triv, \varepsilon \rangle =0$, as $\varepsilon$ is nontrivial.
\end{proof}

\begin{exa}\label{example: counterexample deg3}
For $r=3$ there are locally polyquadratic twists that are not polyquadratic twists. A straightforward computation shows that such an example is given by any pair of nonisomorphic faithful irreducible degree $3$ representations $\varrho$ and $\varrho'$ of the group $G$ with \cite{GAP} identifier $\langle 48,3\rangle$.
This also yields counterexamples for any $r\geq 4$: if $\theta$ is any representation of degree $r-3$, consider $\varrho\oplus \theta$ and $\varrho'\oplus \theta$.
\end{exa}

\subsection{Polyquadratic twists of abelian varieties} In this section, we consider $\ell$-adic representations of abelian varieties. For some prime $\ell$, let 
$$
\varrho_{A,\ell}:G_K \rightarrow \Aut(V_\ell(A))
$$ 
denote the $\ell$-adic representation attached to $A$, where $V_\ell(A)=T_\ell(A) \otimes \mathbb Q$ and  $T_\ell(A)$ is the $\ell$-adic Tate module of $A$ (use the analogous notation for $A'$). For $\p \in \Sigma_K$, let $\Frob_\p$ denote an arithmetic Frobenius at $\p$. From Weil, we know that, for every $\p\nmid \ell$ of good reduction for $A$, the polynomial $\det(1-\varrho_{A,\ell}(\Frob_\p)T)$ is well defined and does not depend on the choice of $\ell$.

\begin{prop}\label{reducing-to-rho}
The abelian varieties $A$ and $A'$ are polyquadratic twists if and only if $\varrho_{A,\ell}$ and~$\varrho_{A',\ell}$ are polyquadratic twists for some $\ell$ (equivalently, for all $\ell$).
Moreover, they are locally polyquadratic twists if and only if $\varrho_{A,\ell}$ and $\varrho_{A',\ell}$ are locally polyquadratic twists for some~$\ell$ (equivalently, for all $\ell$).
\end{prop}
\begin{proof}
The first assertion is an immediate consequence of Faltings' isogeny theorem \cite{Fal83} and Proposition \ref{proposition: polyquadratictwists}, so consider the second assertion. Since $A$ and $A'$ being locally polyquadratic twists does not depend on $\ell$, it suffices to prove the statement for some fixed prime $\ell$. 
By the Chebotarev density theorem it suffices to show the following: if $\p\in \Sigma_K$ is a prime of good reduction for $A$ and $A'$ with $\p\nmid \ell$, then $A_\p$ and $A'_\p$ are polyquadratic twists if and only if
\begin{equation}\label{equation: equality over quad}
\det(1-\varrho_{A,\ell}(\Frob_\p^ 2)T)= \det(1-\varrho_{A',\ell}(\Frob_\p^2)T)
\end{equation}
holds. And this is clear because, by \cite[Thm.\ 1]{Tat66}, \eqref{equation: equality over quad} amounts to saying that the base changes of $A_\p$ and $A'_\p$ to the quadratic extension of the residue field $K(\p)$ of $K$ at $\p$ are isogenous.
\end{proof}

By the above proposition and Remark \ref{rem: locpol implies pol}, we see that if $A$ and $A'$ are polyquadratic twists, then they are locally polyquadratic twists. This may be seen more directly by means of the following remark.

\begin{rem}\label{explanation}
Suppose that $A$ and $A'$ are polyquadratic twists. Then 
$A$ and $A'$ are locally polyquadratic twists at almost all primes of $\Sigma_K$.
Indeed, choose a polyquadratic extension~$L/K$ and an isogeny $f: A'_L\rightarrow A_L$. There exists a finite set $S\subseteq \Sigma_K$ such that for every $\mathfrak p\in \Sigma_K\setminus S$ the abelian varieties $A$ and $A'$ have good reduction at $\mathfrak p$ and, if $\mathfrak P\in \Sigma_L$ lies over $\mathfrak p$, then there is an isogeny $f_{\mathfrak P}: A'_{L, \mathfrak P}\rightarrow A_{L, \mathfrak P}$ making the diagram
\[
  \begin{tikzcd}
    A_L \arrow{r}{f} \arrow[swap]{d}{\mathrm{red}} & A'_L \arrow{d}{\mathrm{red}} \\
    A_{L,\mathfrak P} \arrow[swap]{r}{f_\mathfrak P} & A'_{L,\mathfrak P}
  \end{tikzcd}
\]
commutative.
\end{rem}

\begin{rem}
Let $\varrho$ and $\varrho'$ be as in \S\ref{section: preliminaries}. The  argument of the proof of Proposition \ref{proposition: polyquadratictwists} shows that there exists a normal subgroup $H\subseteq G$ of index $2$ such that $\varrho|_H\simeq \varrho'|_H$ if and only if $\varrho$ and $\varrho'$ admit decompositions as in \eqref{equation: polquadtwist} such that each of the characters $\chi_i$ is either the trivial or the nontrivial character of $G/H$. Hence, if $A$ and $A'$ are abelian varieties that become isogenous over a quadratic extension $L/K$, then each of the characters $\chi_i$ relating their $\ell$-adic representations is either the trivial or the nontrivial character of $\Gal(L/K)$. According to our definition, $A$ and $A'$ are called quadratic twists precisely when all of the $\chi_i$ can be taken to be equal.
\end{rem}

We can use Proposition \ref{reducing-to-rho} to show that several of the notions defined in the introduction in fact coincide in dimension $1$.

\begin{lem}\label{lemma: dimension1} Let $A$ and $A'$ be elliptic curves. The following are equivalent:
\begin{enumerate}[i)]
\item $A$ and $A'$ are quadratic twists. 
\item $A$ and $A'$ are polyquadratic twists.
\item$ A$ and $A'$ are locally quadratic twists.
\item $A$ and $A'$ are locally polyquadratic twists.
\end{enumerate}
\end{lem}

\begin{proof}
Since $i) \Rightarrow ii) \Rightarrow iv)$ and $i) \Rightarrow iii) \Rightarrow iv)$ are obvious, we only need to prove that $iv) \Rightarrow i)$. Suppose that $A$ and $A'$ are locally polyquadratic twists. The fact that $\det\varrho_{A,\ell}=\det\varrho_{A',\ell}=\chi_\ell$, where $\chi_\ell$ denotes the $\ell$-adic cyclotomic character, implies that $A$ and $A'$ must in fact be locally quadratic twists. We can now conclude with Theorem \ref{thm: Ramakrishnan}. 
\end{proof}

In \S\ref{section: proof}, we will make use of the following lemma to prove Theorem \ref{theorem: main}.
 
\begin{lem}\label{lemma: handy}
Let $A$ and $A'$ be abelian surfaces and let $\p$ be a prime of good reduction for $A$ and $A'$. If $A_\p$ and $A'_\p$ are polyquadratic twists and for some prime $\ell$ the traces of $\varrho_{A,\ell}(\Frob_\p)$ and $\varrho_{A',\ell}(\Frob_\p)$ are both zero, then $A_\p$ and $A'_\p$ are quadratic twists.
\end{lem} 
\begin{proof}
Let $\alpha, \overline \alpha, \beta,\overline \beta$ be the eigenvalues of $\varrho_{A,\ell}(\Frob_\p)$. Then there exist $\epsilon,\gamma\in \{ \pm 1\}$ such that
$\epsilon \alpha, \epsilon\overline \alpha, \gamma\beta,\gamma\overline \beta$ are the eigenvalues of $\varrho_{A',\ell}(\Frob_\p)$. The vanishing of the traces yields the equation
$$
\begin{pmatrix}
1 & 1\\
\epsilon & \gamma
\end{pmatrix}
\begin{pmatrix}
\alpha+\overline\alpha\\
\beta+\overline\beta
\end{pmatrix}=\begin{pmatrix}0\\0\end{pmatrix}\,.
$$ 
This implies that $\epsilon=\gamma$ or $\alpha+\overline\alpha=\beta+\overline\beta=0$. Either of the possibilities shows that $A_\p$ and~$A'_\p$ are quadratic twists.
\end{proof}

We conclude this section by recalling a result that will be used multiple times in the proof of Theorem \ref{theorem: main}.
In \S\ref{section: grouptheoretic} we observed that Theorem \ref{thm: Ramakrishnan} does not remain true in general for representations of degree $r\geq 4$. However, the following result asserts that Theorem \ref{thm: Ramakrishnan} remains true in degrees 4 and 6 if one restricts to $\ell$-adic representations attached to abelian surfaces or threefolds.

\begin{thm}[\cite{Fit22}]\label{theorem: Fite} 
Suppose that $A$ and $A'$ have dimension at most $3$. If $\varrho_{A,\ell}$ and $\varrho_{A',\ell}$ are locally quadratic twists, then they are quadratic twists. 
\end{thm}

\subsection{A theorem by Khare and Larsen.} We denote by $\overline A_\p$ the base change of $A_\p$ to the algebraic closure of $K(\p)$, and similarly define $\overline A'_\p$. The following result of Khare and Larsen \cite{KL20} will play an important role in our investigation.

\begin{thm}[Khare--Larsen]\label{theorem: KhareLarsen}
The abelian varieties $A_\Qbar$ and $A'_\Qbar$ are isogenous if and only if for almost all primes $\p\in \Sigma_K$ the abelian varieties $\overline A_\p$ and~ $\overline A'_\p$ are isogenous.
\end{thm}
\begin{proof}
For the ``only if" implication we may reason as in Remark \ref{explanation}. Now suppose that for almost all  $\p$ the abelian varieties $\overline A_\p$ and~ $\overline A'_\p$ are isogenous. We invoke \cite[Thm.\ 1]{KL20}, which asserts: 
\begin{equation}\label{equation: KhareLarsen}
\text{If } \Hom(\overline A_\p,\overline A_\p')\not = 0 \text{ for almost all $\p\in \Sigma_K$, then }\Hom( A_\Qbar,A_\Qbar')\not =0\,.
\end{equation}
The theorem follows from \eqref{equation: KhareLarsen} by induction on the number $n$ of simple isogeny factors of $A_\Qbar$, as we now explain. Note that the assumption implies, in particular, that $A$ and $A'$ have the same dimension. Hence if $n=1$, the existence of a nontrivial homomorphism from $A_\Qbar$ to $A_\Qbar'$ is equivalent to $A_\Qbar$ and $A_\Qbar'$ being isogenous. For general $n$, \eqref{equation: KhareLarsen} implies that $A_\Qbar$ and $A_\Qbar'$ have a common simple isogeny factor. Let $B$ and $B'$ be the complements of this simple isogeny factor in $A_\Qbar$ and $A'_\Qbar$, respectively. Then, for almost all $\p$ in $\Sigma_K$, the reductions of $B$ and $B'$ modulo $\p$ are isogenous. Since the number of simple isogeny factors in $B$ is $n-1$, by induction we find that $B$ and $B'$ are isogenous. Hence so are $A_\Qbar$ and $A'_\Qbar$.
\end{proof} 

\section{Squares of elliptic curves and quaternionic multiplication}\label{section: squares}

This section is devoted to proving Theorem \ref{theorem: main} in the following particular case:

\begin{quote}
Suppose that the abelian surfaces $A$ and $A'$ are locally polyquadratic twists, and that $A_\Qbar$ is either isogenous to the square of an elliptic curve or has quaternionic multiplication. Then $A$ and $A'$ are polyquadratic twists.
\end{quote}

By Theorem \ref{theorem: KhareLarsen}, the assumptions imply that there exists a finite Galois extension $L/K$ such that either:

\begin{itemize}
\item There is an elliptic curve $E$ defined over $L$ such that $A_L\sim E^ 2$ and $A'_L\sim E^ 2$; or
\item $\End(A_L)\otimes \Q \simeq \End(A'_L)\otimes \Q$ is a quaternion algebra.
\end{itemize}
We divide the first case into three subcases, according to whether $E$ has CM or not, and in the former case according to whether the imaginary quadratic field $M$ by which $E$ has CM is contained in $K$ or not.
Notice that by Proposition \ref{reducing-to-rho} it suffices to show that for some $\ell$ the representations $\varrho_{A,\ell}$ and $\varrho_{A',\ell}$ are polyquadratic twists. In some cases, we show that the hypothesis that $A$ and $A'$ are locally polyquadratic twists already implies the a priori stronger hypothesis that $A$ and $A'$ are locally quadratic twists. In these cases, we can conclude by directly applying Theorem \ref{theorem: Fite}.

\subsection{Non CM or quaternionic multiplication}
By \cite[Thm.\ 1.1]{FG22} (see also \cite[Thm.\ 4.5 (i)]{Fit22} for a restatement of this result in our situation) and after enlarging $L/K$ if necessary, there exist a number field $F$, Artin representations $\theta, \theta':\Gal(L/K)\rightarrow \GL_2(F)$, a prime $\ell$ totally split in $F$, and strongly absolutely irreducible $F$-rational $\ell$-adic representations $\varrho, \varrho'$ of $G_K$ of degree $2$ such that 
$$
\varrho_{A,\ell} \simeq \theta\otimes \varrho\,, \qquad \varrho_{A',\ell} \simeq \theta'\otimes \varrho'\,,\qquad \varrho|_L \simeq \varrho'|_L\,.
$$
In particular, there exists a character $\chi$ of $\Gal(L/K)$ such that $\varrho'\simeq \chi \otimes \varrho$.
Let $\alpha_{1,\p},\alpha_{2,\p}$ 
be the eigenvalues of $\varrho(\Frob_\p)$; $\beta_{1,\p},\beta_{2,\p}$ the eigenvalues of $\theta(\Frob_\p)$; and $\gamma_{1,\p},\gamma_{2,\p}$ the eigenvalues of $\chi\otimes\theta'(\Frob_\p)$. By \cite[Lem. 4.7]{Fit22}, there exists a density $1$ subset $\Sigma$ of $\Sigma_K$ of primes of good reduction for $A$ and $A'$ such that for every $\p\in \Sigma$ the quotient $\alpha_{1,\p}/\alpha_{2,\p}$ is not a root of unity, and there exist $\epsilon_{ij,\p}\in \{ \pm 1\}$ such that
$$
\prod_{i,j=1}^2(1- \alpha_{i,\p}\beta_{j,\p}T)=\prod_{i,j=1}^2(1-\epsilon_{ij,\p} \alpha_{i,\p}\gamma_{j,\p}T)\,.
$$ 
Hence, for every $\p\in \Sigma$ and for $i\in \{1,2\}$, we have 
$$
\prod_{j=1}^2(1- \beta_{j,\p}T)=\prod_{j=1}^2(1-\epsilon_{ij,\p} \gamma_{j,\p}T)\,.
$$
By the Chebotarev density theorem, this implies that $\theta$ and $\chi \otimes \theta'$ are locally polyquadratic twists and hence polyquadratic twists by Theorem \ref{thm: locpolytwists}. By Remark \ref{used} and the fact that $\varrho_{A',\ell}\simeq \theta'\otimes \chi\otimes \varrho$ and $\varrho_{A,\ell}\simeq \theta \otimes \varrho$, we conclude that $\varrho_{A,\ell}$ and $\varrho_{A',\ell}$ are polyquadratic twists. 

\subsection{CM over $K$}\label{section: CMoverK}
By \cite[Thm. 1.1]{FG22} (see also  \cite[Thm.\ 4.5 (ii)]{Fit22}), after enlarging $L/K$ if necessary, there exist a Galois number field $F$ containing $M$, Artin representations $\theta, \theta':\Gal(L/K)\rightarrow \GL_2(F)$, a prime $\ell$ totally split in $F$, and continuous $F$-rational $\ell$-adic characters $\psi, \psi'$ of $G_K$ such that
\begin{equation}\label{equation: ladicCM}
\varrho_{A,\ell}\simeq (\theta \otimes \psi) \oplus(\overline \theta \otimes  \overline\psi)\,, \quad \text{and}\quad \varrho_{A',\ell}\simeq (\theta' \otimes \psi') \oplus( \overline\theta' \otimes \overline\psi')\,.
\end{equation}
Here $\overline\theta$ (resp.\ $\overline\theta'$, $\overline \psi$, $\overline \psi'$) stands for the ``complex conjugate" of $\theta$ (resp.\ $\theta'$, $\psi$, $\psi'$). By this, more precisely we mean the following. Fix a complex conjugation $\overline \cdot $ in $\Gal(F/\Q)$, that is, a lift of the nontrivial element of $\Gal(M/\Q)$. That $\psi$ is an $F$-rational $\ell$-adic character implies that $\psi(\Frob_\p)\in F$ for almost all $\p \in \Sigma_K$. For any such $\p$, the $\ell$-adic character $\overline \psi$ satisfies $\overline \psi(\Frob_\p)=\overline{\psi(\Frob_\p)}$. Similar descriptions hold for $\overline\theta'$, $\overline \psi$, and $\overline \psi'$.

After reordering~$\psi$ and~$\overline\psi$ if necessary, we may assume that $\psi|_L\simeq \psi'|_L$. Hence there exists a character $\varphi$ of $\Gal(L/K)$ such that $\psi'\simeq \varphi \psi$. Let $\beta_{1,\p},\beta_{2,\p}$ be the eigenvalues of $\theta(\Frob_\p)$ and $\gamma_{1,\p},\gamma_{2,\p}$ the eigenvalues of $\varphi\otimes\theta'(\Frob_\p)$. By \cite[Lem. 4.7]{Fit22}, there exists a density $1$ subset $\Sigma$ of $\Sigma_K$ of primes of good reduction for $A$ and $A'$ such that for every $\p\in \Sigma$ the quotient $\psi(\Frob_{\p})/\overline\psi(\Frob_{\p})$ is not a root of unity, and there exist $\epsilon_{j,\p},\delta_{j,\p}\in \{ \pm 1\}$ such that
$$
\prod_{j=1}^2(1- \psi(\Frob_{\p})\beta_{j,\p}T)(1- \overline\psi(\Frob_{\p})\overline\beta_{j,\p}T)=\prod_{j=1}^2(1-\epsilon_{j,\p} \psi(\Frob_{\p})\gamma_{j,\p}T)(1-\delta_{j,\p} \overline \psi(\Frob_{\p})\overline\gamma_{j,\p}T)\,.
$$  
Hence for every $\p\in \Sigma$ we have 
$$
\prod_{j=1}^2(1- \beta_{j,\p}T)=\prod_{j=1}^2(1-\epsilon_{j,\p} \gamma_{j,\p}T)\quad\text{and}\quad 
\prod_{j=1}^2(1- \overline\beta_{j,\p}T)=\prod_{j=1}^2(1-\delta_{j,\p} \overline\gamma_{j,\p}T)\,.
$$
By the Chebotarev density theorem, each of the above equations implies that $\theta$ and $\varphi\otimes \theta'$ are locally polyquadratic twists and we conclude as in the previous case.

\subsection{CM not over $K$}
Given a representation $\varrho$ of $G_L$ for a finite extension $L/K$, we write $\Ind^K_L(\varrho)$ to denote the induction of $\varrho$ from $G_L$ to $G_{K}$. In this section, to shorten the notation, we omit the subindex and superindex fields in $\Ind_{KM}^K(\varrho)$. The last statement of \cite[Thm. 4.5 (ii)]{Fit22} implies that there exist Artin representations $\theta,\theta'$ of $G_{KM}$, an $\ell$-adic character $\psi$ of~$G_{KM}$, and a character $\varphi$ of $\Gal(L/KM)$ satisfying 
$$
\varrho_{A,\ell}\simeq \Ind(\theta\otimes \psi)\,,\qquad \varrho_{A',\ell}\simeq \Ind(\varphi \otimes \theta'\otimes \psi)\,.
$$
Moreover, by the construction of $\theta$ and $\psi$ (see \cite[\S2]{FG22} and especifically \cite[Lem. 2.10]{FG22}), one has that if $\xi$ is a constituent of~$\theta$, then 
\begin{equation}\label{equation: compconjprop}
\Ind(\xi\otimes \psi)|_L\simeq \xi\otimes \psi \oplus \overline \xi\otimes \overline \psi\,.
\end{equation}
By the previous case, one of the following holds:
\begin{enumerate}[i)]
\item There exists a quadratic character $\chi$ of $\Gal(L/KM)$ such that $\theta' \otimes \varphi \simeq \chi \otimes \theta$. 
\item There exist characters $\xi_1,\xi_2$ and quadratic characters $\chi_1,\chi_2$ of $\Gal(L/KM)$ such that 
$$
\theta\simeq \xi_1 \oplus \xi_2\,,\qquad \theta' \otimes \varphi \simeq \chi_1 \xi_1 \oplus \chi_2 \xi_2\,.
$$
\end{enumerate}

We first consider case ii). Then, we have
$$
\varrho_{A,\ell}\simeq \Ind(\xi_1 \psi)\oplus\Ind(\xi_2 \psi)\,,\qquad \varrho_{A',\ell}\simeq \Ind(\chi_1\xi_1 \psi)\oplus \Ind(\chi_2\xi_2 \psi)\,.
$$
Observe that $\Ind(\xi_i \psi)$ and $\Ind(\chi_i\xi_i \psi)$ are locally quadratic twists for $i=1,2$ (this can be seen by comparing traces on an element $s\in G_K$: if $s\in G_L$, then use \eqref{equation: compconjprop}, and if $s\not\in G_L$, then notice that both traces are $0$). By Theorem~\ref{thm: Ramakrishnan}, they are quadratic twists, which finishes the proof in this case.

In case i), we claim that $A$ and $A'$ are locally quadratic twists, which is enough for our purposes in virtue of Theorem \ref{theorem: Fite}. Let $\p \in \Sigma_K$ be a prime of good reduction for $A$ and $A'$ of absolute residue degree $1$. If $\p$ is split in $KM$, it follows from \eqref{equation: ladicCM} that the reductions $A_\p$ and $A'_\p$ are quadratic twists. If $\p$ is inert in $KM$, then the same conclusion is attained by using the lemma below.

\begin{lem}\label{lemma: inert primes}
Let $A$ and $A'$ be abelian surfaces $\Qbar$-isogenous to the square of an elliptic curve with CM, say by an imaginary quadratic field $M$. Then, for every $\p \in \Sigma_K$ of good reduction for $A$ and $A'$, inert in $KM$, and of absolute residue degree 1, the reductions $A_\p$ and $A'_\p$ are isogenous if and only if they are polyquadratic twists.
\end{lem}

\begin{proof}
Let $\p\in \Sigma_K$ be as in the statement, and let $p$ denote its absolute norm. 
Define the set of polynomials
$$
S:=\{(1-T^2)^2\,, 1-T^2+T^4\,,1+T^4\,,  1+T^2+T^4\,,  (1+T^2)^2\}\,.
$$
The fact that such a $\p$ is of supersingular reduction for $A$, together with the Weil bounds, implies that $L_\p(A,p^{-1/2}T)\in S$. We similarly have $L_\p(A',p^{-1/2}T)\in S$. Consider the set
$$
R:=\{(1-T)^4\,,(1-T+T^2)^2\,,(1-T^2)^2\,, (1+T+T^2)^2\,, (1+T)^ 4  \}\,,
$$
and observe that the map $\Phi:S\rightarrow R$ defined by 
$P(T) \mapsto \Res_Z(P(Z),Z^ 2-T)$ is a bijection.
Notice that $A_\p$ and $A'_\p$ are polyquadratic twists if and only if the polynomials $L_\p(A,p^{-1/2}T)$ and $L_\p(A',p^{-1/2}T)$ have the same image under $\Phi$. Since $\Phi$ is a bijection, this is equivalent to their equality, and we conclude by \cite[Thm. 1]{Tat66}.
\end{proof}

\section{Proof of Theorem \ref{theorem: main}}\label{section: proof}

In this section we complete the proof of Theorem \ref{theorem: main}. 

\subsection{Background results.} For an abelian surface $A$ defined over $K$, denote by $\End^0(A)$ the endomorphism algebra $\End(A)\otimes \Q$, and by $\Xc(A)$ the \emph{absolute type} of $A$, that is, the $\R$-algebra $\End(A_\Qbar)\otimes \R$. We borrow from \cite[\S4.1]{FKRS12} the labels $\Ab,\dots,\Fb$ for the different possibilities for $\Xc(A)$. These are in bijection with the possibilities for the connected component of the identity of the \emph{Sato--Tate group} of $A$, denoted $\ST(A)$. This is a closed real Lie subgroup of $\USp(4)$, only defined up to conjugacy. It captures important arithmetic information of $A$ and it is conjectured to predict the limiting distribution of the Frobenius elements attached to~$A$. See \cite[\S2]{BK15} or \cite[\S2]{FKRS12} for its definition in our context; see \cite[Chap. 8]{Ser12} for a conditional definition in a more general context. We will use the notation settled in \cite[\S3]{FKRS12} for Sato--Tate groups of abelian surfaces.

It follows from Theorem \ref{theorem: Fite} (as explained in \cite[Rem.\ 2.10]{Fit22}) that two abelian surfaces which are locally quadratic twists share the same Sato--Tate group. There are however examples of abelian surfaces that are locally polyquadratic twists with distinct Sato--Tate groups. Indeed, let $E$ be an elliptic curve defined over $K$ without CM and $\chi$ a nontrivial quadratic character of $G_K$. Let $E_\chi$ denote the twist of $E$ by $\chi$. Then $E\times E$ has Sato--Tate group~$\stgroup[E_1]{1.4.E.2.1a}$, while $E\times E_\chi$ has Sato--Tate group $\stgroup[J(E_1)]{1.4.E.2.1b}$. Despite the existence of these examples, the Sato-Tate group is preserved under locally polyquadratic twist in most of the cases. This is essentially the first step in the proof of Theorem \ref{theorem: main} in the cases considered in this section.

We will repeatedly use the following well known lemma, whose proof may be found split between the proofs of \cite[Cor.\ 2.5]{Fit22} and \cite[Cor.\ 2.7]{Fit22}.

\begin{lem}\label{lemma: trivialends}
Let $\varrho$ and $\varrho'$ be strongly absolutely irreducible representations of $G_K$. If there exists a finite extension $L/K$ such that $\varrho|_L\simeq \varrho'|_L$, then there exists a finite order character $\chi$ of $G_K$ such that $\varrho'\simeq \chi \otimes \varrho$. In particular, if $B$ and $B'$ are abelian varieties defined over $K$ such that $B_\Qbar$ and $B'_\Qbar$ are isogenous and such that $\Xc(B)\simeq \Xc(B')\simeq \R$, then $B$ and $B'$ are quadratic twists. 
\end{lem}

\subsection{Proof strategy.} Until the end of \S\ref{section: proof}, let $A$ and $A'$ be abelian surfaces that are locally polyquadratic twists. By Theorem~\ref{theorem: KhareLarsen},~$A$ and~$A'$ are geometrically isogenous. In particular, they have the same absolute type, and the proof is carried out by cases depending on each absolute type.

Before delving into the details of the proof for each case, we make some general remarks. As we did in \S\ref{section: squares}, for those absolute types for which the a priori weaker condition of being locally polyquadratic twists implies being locally quadratic twists, we prove so, and then we conclude by applying Theorem \ref{theorem: Fite}. In \S\ref{section: bcases} we treat the absolute types for which the proof is straightforward with the results obtained so far. Absolute types $\Db$ and $\Bb$ are respectively considered in \S\ref{section: D} and \S\ref{section: B}. 

The absolute type $\Db$ divides into the geometrically irreducible and the reducible case. The latter is handled with ideas similar to those used in \S\ref{section: bcases}; both for the geometrically irreducible case of type $\Db$ and type $\Bb$, the proof follows a similar pattern. The first step is to show that $A$ and $A'$ have the same Sato--Tate group. This follows quite directly from results by Shimura in the geometrically irreducible case of type $\Db$ and from Lemma \ref{lemma: same subcase} for type $\Bb$. One then shows the validity of the local-global principle in the case that the (common) Sato--Tate group is connected. The final step consists on retrieving the general case from the connected case.

\subsection{Absolute types $\Eb$, $\Fb, \Ab, \Cb$.}\label{section: bcases}

If $A$ has absolute type $\Eb$ or $\Fb$, namely, $\Xc(A)\simeq \M_2(\R)$ or $\M_2(\C)$,
then Theorem \ref{theorem: main} was proven in \S\ref{section: squares}.
If $A$ has absolute type $\Ab$, equivalently, $\Xc(A)\simeq \R$, then~$A$ and~$A'$ are quadratic twists by Lemma \ref{lemma: trivialends}. Now suppose that $A$ has absolute type $\Cb$. By Lemma \cite[Lem.\ 4.13]{Fit22}, $A$ is isogenous to the product of an elliptic curve $E_1$ without CM and an elliptic curve $E_2$ with CM. We similarly define~$E_1'$ and~$E_2'$ for~$A'$.
By Theorem \ref{theorem: KhareLarsen},~$E_{1}$ and~$E'_{1}$ (respectively, $E_{2}$ and $E'_{2}$) are geometrically isogenous, and hence, by Lemma \ref{lemma: trivialends}, $E_1$ and~$E_1'$ are quadratic twists. Then Remark \ref{remark: factors} implies that $E_2$ and $E_2'$ are locally polyquadratic twists, and hence quadratic twists by Lemma \ref{lemma: dimension1}. Thus $A$ and $A'$ are polyquadratic twists.

\subsection{Absolute type $\Db$}\label{section: D}

Suppose that $A$ has absolute type $\Db$, namely, $\Xc(A)\simeq \C\times \C$. Then precisely one of the following two conditions hold:
\begin{enumerate}
\item[D.1)] $A$ is isogenous to the product of two nonisogenous elliptic curves $E_1$ and $E_2$ with CM; 
\item[D.2)] $A_\Qbar$ has CM by a quartic CM field.
\end{enumerate}
Suppose that $A$ falls in case D.1. By \cite[Lem.\ 4.13]{Fit22} and Theorem \ref{theorem: KhareLarsen}, there exist elliptic curves $E'_1$ and $E'_2$ such that $A'$ is isogenous to $E'_1\times E'_2$. After possibly reordering $E_1$ and~$E_2$, and using \cite[Lem.\ 4.14]{Fit22}, we may assume that $E_1$ and $E_1'$ are locally polyquadratic twists. Hence they are quadratic twists by Lemma \ref{lemma: dimension1}. As the same holds for $E_2$ and $E_2'$, we conclude that $A$ and $A'$ are polyquadratic twists.
 
Now suppose that $A$ falls in case D.2. Denote the CM field by $M$ and its reflex field by $M^*$.  As explained in \cite[\S5.1.1]{Fit22} (see the references therein), precisely one of the following three cases occurs:
\begin{enumerate}[i)]
\item $M^*\subseteq K$ and $\End^0(A) \simeq M$;
\item $KM^*/K$ is quadratic and $\End^0(A)$ is a real quadratic field;
\item $KM^*/K$ is cyclic of degree $4$ and $\End^0(A)\simeq \Q$. 
\end{enumerate}
Moreover, which case occurs depends on whether $A$ has Sato--Tate group $\stgroup[F]{1.4.D.1.1a}$, $\stgroup[F_{ab}]{1.4.D.2.1a}$ or $\stgroup[F_{ac}]{1.4.D.4.1a}$, respectively. 
Notice that, as it is clear from the case distinction, $A$ and $A'$ fall in the same case, and in particular have the same Sato--Tate group.
 
Suppose that $A$ and $A'$ fall in case i). We claim that there exists a rational prime $\ell$ inert in~$M$. This follows from the fact that either $M/\Q$ is cyclic or it is not normal with normal closure the dihedral group of 8 elements, as shown in \cite[p. 65]{Shi98} (our claim would not hold if $\Gal(M/\Q)$ were allowed to have order 2). So let $\ell$ be a prime inert in~$M$, let $\lambda$ be the prime of~$M$ lying over $\ell$, and let $M_\lambda$ denote the completion of $M$ at $\lambda$. Then $V_\ell(A)$ is an $M_\lambda$-module of dimension $1$, and hence strongly absolutely irreducible. Therefore there exists a character $\varphi$ of $G_K$ such that $\varrho_{A',\ell}\simeq \varphi \otimes \varrho_{A,\ell}$. Since in this case the Sato--Tate group is $\stgroup[F]{1.4.D.1.1a}$, we can apply \cite[Prop.\ 2.11]{Fit22}, which implies that $A$ and $A'$ are quadratic twists. 

Now suppose that $A$ and $A'$ fall in case ii) or iii). We choose a prime $\ell$ totally split in $M$, so that the four distinct embeddings $\lambda_i\colon M \hookrightarrow \Qbar_\ell$, for $i=1,\dots,4$, take values in $\Q_\ell$.  
Define 
$$
V_{\lambda_i}(A):=V_\ell(A)\otimes_{M\otimes \Q_\ell,\lambda_i}\Q_\ell\,,
$$
where $\Q_\ell$ is being regarded as an $M\otimes \Q_\ell$-module via~$\lambda_i$. It is a 1-dimensional vector space over $\Q_\ell$. We equip it with an action of $G_K$ by letting this group act naturally on $V_\ell(A)$ and trivially on $\Q_\ell$. 

Let $n=[\End^0(A):\Q]$. By \cite[(5.2)]{Fit22}, we have an isomorphism
\begin{equation}\label{equation: Tate module as induction}
V_\ell(A)\simeq \bigoplus_{i=1}^n \Ind^K_{KM^*}(V_{\lambda_i})\,,
\end{equation}
where we assume that $\lambda_1,\dots,\lambda_4$ have been ordered so that the restrictions of $\lambda_1,\dots,\lambda_n$ to $\End^0(A)\subseteq M$ are all distinct. 

Let $\p\in \Sigma_K$ be a prime of good reduction for $A$ and $A'$. If $\p$ is totally split in $KM^*$, then $A_\p$ and~$A_\p'$ are quadratic twists by case i). If $\p$ is not totally split in $KM^*$, then \eqref{equation: Tate module as induction} implies that $\Tr(\varrho_{A,\ell}(\Frob_\p))=\Tr(\varrho_{A',\ell}(\Frob_\p))=0$. Hence we can apply Lemma \ref{lemma: handy} which implies that~$A_\p$ and~$A_\p'$ are quadratic twists. It follows that $\varrho_{A,\ell}$ and $\varrho_{A',\ell}$ are locally quadratic twists, and hence they are quadratic twists by Theorem \ref{theorem: Fite}. 

\subsection{Absolute type $\Bb$}\label{section: B}
Suppose that $A$ has absolute type $\Bb$, which means that $\Xc(A)\simeq \R\times \R$. Then, as explained in \cite[\S5.1.2]{Fit22}, there exists a prime $\ell$ so that precisely one of the following two conditions holds:
\begin{enumerate}[(i)]
\item[B.1)] $\varrho_{A,\ell}\simeq \varrho_1 \oplus \varrho_2$, where $\varrho_1,\varrho_2:G_K\rightarrow \GL_2(\Q_\ell)$ are strongly absolutely irreducible $\ell$-adic representations that do not become isomorphic after restriction to $G_L$ for any finite extension $L/K$;
\item[B.2)] $\varrho_{A,\ell}\simeq \Ind^L_K(\varrho)$, where $L/K$ is quadratic and $\varrho:G_L\rightarrow \GL_2(\Q_\ell)$ is a strongly absolutely irreducible $\ell$-adic representation.
\end{enumerate}

Whether B.1 or B.2 holds depends on whether $\ST(A)$ is isomorphic to $\stgroup[\SU(2)\times \SU(2)]{1.4.B.1.1a}$ or $\stgroup[N(\SU(2)\times \SU(2))]{1.4.B.2.1a}$.

\begin{lem}\label{lemma: same subcase}
If $A$ falls in case B.1 (respectively, B.2), then so does $A'$.
\end{lem}

\begin{proof}
Let $\varrho$ be the tautological representation of $\ST(A)$, and similarly define $\varrho'$ for $\ST(A')$. Let $\theta$ denote $\Sym^2 \varrho - \wedge^ 2\varrho$ and let $\theta_\ell$ denote $\Sym^2 \varrho_{A,\ell} - \wedge^ 2\varrho_{A,\ell}$. Define similarly $\theta'$ and $\theta_\ell'$. The same argument used in the proof of \cite[Thm.\ 1]{Saw16} shows that the virtual multiplicity of the trivial representation in $\theta$ coincides with the virtual multiplicity of the $\ell$-adic cyclotomic character $\chi_\ell$ in $\theta_\ell$.
By Lemma~\ref{lemma: gpcharac}, we have
$$
\langle \triv, \theta\rangle = \langle \chi_\ell,\theta_\ell\rangle=\langle \chi_\ell,\theta_\ell'\rangle=\langle \triv,\theta'\rangle\,.
$$
Suppose that the statement of the lemma were false. Without loss of generality we may assume that $A$ falls in case B.1 and $A'$ in case B.2. However, \cite[Table 8]{FKRS12} yields
$$
\langle \triv, \theta\rangle=\langle \triv,\varrho\otimes \varrho\rangle -2 \langle \triv,\wedge^ 2\varrho\rangle=-2\,,\quad 
\langle \triv, \theta'\rangle=\langle \triv,\varrho'\otimes \varrho'\rangle -2 \langle \triv,\wedge^ 2\varrho'\rangle=-1\,,
$$
which gives a contradiction.
\end{proof}

\begin{lem}\label{lemma: ladic constituents}
If $A$ falls in case B.1, then $\wedge^2 \varrho_{A,\ell}$ has no $1$-dimensional constituent of the form $\psi\chi_\ell$, where $\psi$ is a nontrivial finite order character and $\chi_\ell$ is the $\ell$-adic cyclotomic character. 
\end{lem}

\begin{proof}
Write $\varrho_{A,\ell}$ as $\varrho_1\oplus \varrho_2$, where $\varrho_1,\varrho_2$ are as in B.1. By \cite[Thm.\ 4.3.1]{Rib76}, we have $\det(\varrho_i)=\chi_\ell$ and hence
$$
\wedge^2 \varrho_{A,\ell}\simeq \chi_\ell \oplus (\varrho_1 \otimes \varrho_2) \oplus \chi_\ell\,.
$$
Suppose that $\varrho_1 \otimes \varrho_2$ has a 1-dimensional constituent of the form $\psi\chi_\ell$ and let $L/K$ be the field extension cut out by $\psi$. Then the subspace affording $\psi$ is contained in 
$$
(\varrho_1\otimes \varrho_2\otimes \chi_\ell^ {-1})^{G_L}\simeq (\varrho_1^ \vee\otimes \varrho_2)^{G_L}\simeq \Hom_{G_L}(\varrho_1,\varrho_2)\,.
$$
This contradicts that $\varrho_1$ and $\varrho_2$ are strongly absolutely irreducible and $\varrho_1|_L\not \simeq \varrho_2|_L$. 
\end{proof}

Suppose that $A$ and $A'$ fall in case B.1. Write accordingly $\varrho_{A,\ell}$ as $\varrho_1\oplus \varrho_2$ and $\varrho_{A',\ell}$ as $\varrho_1'\oplus \varrho_2'$.
After possibly reordering $\varrho_1$ and $\varrho_2$, Lemma~\ref{lemma: trivialends} shows that there exist finite order characters~$\chi_1$ and~$\chi_2$ of $G_K$ such that $\varrho_1'\simeq \chi_1\otimes \varrho_1$ and $\varrho_2'\simeq \chi_2\otimes \varrho_2$. Since
$$
\wedge^2 \varrho_{A',\ell}\simeq \chi_1^2\chi_\ell \oplus \chi_1\chi_2 (\varrho_1\otimes \varrho_2) \oplus \chi_2^2\chi_\ell\,,
$$
Lemma~\ref{lemma: ladic constituents} implies that the characters $\chi_1$ and $\chi_2$ must be  quadratic. Hence $A$ and $A'$ are polyquadratic twists.

Now suppose that $A$ and~$A'$ fall in case B.2. Then we write $\varrho_{A,\ell}\simeq \Ind_K^L(\varrho)$ and $\varrho_{A',\ell}\simeq \Ind_K^{L'}(\varrho')$, where $L/K$ and $L'/K$ are quadratic extensions, and $\varrho$ and $\varrho'$ are strongly absolutely irreducible representations. 
Observe that $A_L$ falls in case B.1. Then Lemma \ref{lemma: same subcase} implies that~$A'_L$ also falls in case B.1. This can only occur if $L=L'$.

Let $\tau$ denote an element of $G_K$ projecting onto the nontrivial element of $\Gal(L/K)$. Let $\varrho^\tau$ denote the representation of $G_L$ defined by $\varrho^\tau(s)=\varrho(\tau s \tau)$ for every $s\in G_L$. Then we have 
\begin{equation}\label{equation: overL1}
\varrho_{A,\ell}|_L\simeq \varrho \oplus \varrho^ \tau\,,\qquad \varrho_{A',\ell}|_L\simeq \varrho' \oplus \varrho'^ \tau\,.
\end{equation}
After possibly reordering $\varrho$ and $\varrho^\tau$, by the B.1. case of the proof, there exist quadratic characters $\chi$ and $\psi$ of $G_L$ such that 
\begin{equation}\label{equation: overL2}
\varrho'\simeq \chi \otimes \varrho\,,\qquad \varrho'^\tau\simeq \psi \otimes \varrho^\tau\,.
\end{equation} 
We claim that the characters $\chi$ and $\psi$ coincide. Assuming the claim, the proof of Theorem \ref{theorem: main} takes the following steps. Let $\p\in \Sigma_K$ be a prime of good reduction for $A$ and~$A'$. If $\p$ is split in $L/K$, the isomorphisms \eqref{equation: overL1} and \eqref{equation: overL2} show~ that~$A_\p$ and~$A'_\p$ are quadratic twists. The same holds, by Lemma \ref{lemma: handy}, if $\p$ is inert in $L/K$. 
Then by Theorem~\ref{theorem: Fite}, $A$ and $A'$ are quadratic twists.

We now turn to prove the claim. There are two cases to distinguish: either $A$ is an abelian surface such that $\End^0(A_L)$ is isomorphic to a real quadratic field $F$ or $A_L$ is isogenous to the product $E\times E^\tau$, where $E$ is an elliptic curve defined over $L$ without complex multiplication, and $E^\tau$ is the Galois conjugate of $E$ by $\tau$. 

In the first case, $\ell$ is split in $F$. Let $\overline{\,\cdot\,}$ denote the nontrivial element of $\Gal(F/\Q)$. After possibly renaming the primes $\lambda$ and $\overline \lambda$ of $F$ lying over $\ell$, we have that $\varrho$ (resp. $\varrho^ \tau$) is the $F$-rational representation afforded by 
$$
V_\lambda(A):=V_\ell(A)\otimes_{F\otimes\Q_\ell,\lambda}\Q_\ell\qquad  \left(\text{resp. $V_{\overline \lambda}(A):=V_\ell(A)\otimes_{F\otimes\Q_\ell,\overline\lambda}\Q_\ell$}\right). 
$$
From \cite[Lem. 2.10]{FG22}, we deduce that $\varrho^\tau\simeq \overline \varrho$, where $\overline \varrho$ is the $F$-rational $\ell$-adic representation characterized by the property that, for almost all $\p \in \Sigma_L$, the traces of $\overline\varrho(\Frob_\p)$ and $\varrho(\Frob_\p)$ are Galois conjugates in $F$. Since $\chi$ is quadratic and takes rational values, using the first isomorphism of \eqref{equation: overL2}, we find 
$$
\varrho'^\tau \simeq \overline \varrho'\simeq \chi\otimes \overline \varrho \simeq \chi \otimes \varrho^\tau\,,
$$
which together with the second isomorphism of \eqref{equation: overL2} concludes the proof of the claim in the first case.

In the second case, by Theorem \ref{theorem: KhareLarsen}, there is an elliptic curve $E'$ defined over $L$ such that $A'_L$ is isogenous to $E'\times E'^ \tau$, and we may take 
$$
\varrho=\varrho_{E,\ell}\,,\quad \varrho^\tau=\varrho_{E^\tau,\ell}\,,\quad\varrho'=\varrho_{E',\ell}\,,\quad \varrho'^\tau=\varrho_{E'^\tau,\ell}\,.
$$
Then, the first isomorphism of \eqref{equation: overL2} expresses the fact that $E'$ is the quadratic twist of $E$ by $\chi$. Choosing Weierstrass equations $y^2=x^3+ax+b$ and $dy^2=x^3+ax+b$, with $a,b,d\in \mathcal O_L$, for~$E$ and~$E'$ respectively, we see that for almost all $\p\in \Sigma_L$ whether $\chi(\Frob_\p)$ is $1$ or $-1$ depends on whether $d$ is a square or not in $L(\p)^\times$. The same reasonning using $E^\tau$ and $E'^\tau$ instead, shows that for almost all $\p\in \Sigma_L$ whether $\psi(\Frob_\p)$ is $1$ or $-1$ depends on whether $\tau(d)$ is a square or not in $L(\p)^\times$. Since $d$ is a square in $L(\p)^\times$ if and only if so is $\tau(d)$ in $L(\tau(\p))^\times\simeq L(\p)^\times$, we deduce that $\chi$ and $\psi$ coincide, as desired.

\section{A counterexample in dimension 3}\label{section: counterexample}

Let $\varrho$ and $\varrho'$ be as in Example \ref{example: counterexample deg3}, namely, a pair of nonisomorphic faithful irreducible degree~$3$ representations of the group $G$ with \cite{GAP} identifier $\langle 48,3\rangle$. They are realizable over the quadratic ring $\Z[i]$. Let $\href{https://www.lmfdb.org/NumberField/12.4.213838914125824000000.1}{F}$ be the number field $\Q[T]/(f)$, where
$$
f(T)=T^{12} + 2T^{10} - 82T^8 + 50T^6 + 595T^4 + 500T^2 + 25\,.
$$
Let $\tilde F$ be the Galois closure of $\href{https://www.lmfdb.org/NumberField/12.4.213838914125824000000.1}{F}$. According to \cite{LMFDB}, the \cite{GAP} identifier of $\Gal(\tilde F/\Q)$ is $\langle 48,3\rangle$. A straightforward computation shows that $\tilde F$ does not contain the quadratic field $K=\Q(i)$. Let $L$ denote the compositum of $\tilde F$ and $K$, and choose an isomorphism between $G$ and $\Gal(L/K)$. Use this isomorphism to regard $\varrho$ and $\varrho'$ as Artin representations of the group $\Gal(L/K)$. 

Let $E$ be an elliptic curve with CM by $\Z[i]$ defined over $K$, and let~$A$ and~$A'$ be the abelian varieties $\varrho\otimes_{\Z[i]} E$ and $\varrho'\otimes_{\Z[i]} E$ 
defined over $K$ as described in \cite[\S2]{Mil72}. By \cite[Thm.\ 2.2\ (iii)]{MRS07}, we have that there exists a $K$-rational $\ell$-adic character $\psi$ of $G_K$ such that
\begin{equation}\label{equation: ldeccounter}
\varrho_{A,\ell}\simeq (\varrho \otimes \psi)\oplus (\overline \varrho \otimes \overline \psi)\,,\quad \text{and}\quad
\varrho_{A',\ell}\simeq (\varrho' \otimes \psi)\oplus (\overline \varrho' \otimes \overline \psi)\,. 
\end{equation}
Observe that this is compatible with \eqref{equation: ladicCM}: by our choices of $A$ and $A'$, the Artin representations~$\theta$ and~$\theta'$ in that formula can be taken as $\varrho$ and $\varrho'$. 
From \eqref{equation: ldeccounter}, we see that the fact that $\varrho$ and $\varrho'$ are locally polyquadratic twists implies that so are $A$ and $A'$.

We claim however that the threefolds $A$ and $A'$ are not polyquadratic twists.
If the contrary were true, then as in \S\ref{section: CMoverK} the existence of a density 1 subset $\Sigma$ of $\Sigma_K$ of primes of good reduction for $A$ and $A'$ such that for every prime $\p \in \Sigma$ the quotient $\psi(\Frob_\p)/\overline \psi(\Frob_\p)$ is not a root of unity, would imply, via the Chebotarev density theorem, that $\varrho\otimes \psi$ and $\varrho'\otimes \psi$ are polyquadratic twists. This contradicts the fact that $\varrho$ and $\varrho'$ are not polyquadratic twists.

\end{document}